\documentclass[12pt]{article}
\usepackage{amsmath,amsopn,amssymb,amsthm,amsfonts}
\usepackage[T2A]{fontenc}
\usepackage[cp1251]{inputenc}
\usepackage{longtable}

\newtheorem{theorem}{Theorem}[section]
\newtheorem{corollary}[theorem]{Corollary}
\newtheorem{proposition}[theorem]{Proposition}
\newtheorem{definition}[theorem]{Definition}

\newcommand{\w}{\omega}
\newcommand\g{{\mathfrak g}}
\newcommand\G{{\mathfrak G}}

\newcommand\ag{{\mathfrak a}}

\begin{document}

{\bf \large
\centerline{N.~K.~Smolentsev}

\vspace{3mm}
\centerline{Left-invariant almost para-K\"{a}hler  structures}
\centerline{on six-dimensional nilpotent Lie groups}}
\vspace{3mm}

\begin{abstract}
In this paper, we consider left-invariant para-complex structures on six-dimensional nilpotent Lie groups. A complete list of six-dimensional nilpotent Lie groups that admit para-K\"{a}hler structures is obtained, explicit expressions for para-complex structures are found, and curvature properties of associated para-K\"{a}hler metrics are investigated. It is shown that paracomplex structures are nilpotent and the corresponding para-K\"{a}hler metrics are Ricci-flat.
\end{abstract}

\section{Introduction} \label{Intro}
As is known \cite{BG}, nilpotent Lie groups, with the exception of the Abelian case, do not admit left-invariant positive-definite K\"{a}hler metrics. However, pseudo-K\"{a}hlerian (i.e. with a pseudo-Riemannian metric) structures may exist.
In \cite{CFU}, a complete list of 13 classes of non-commutative six-dimensional nilpotent Lie groups was obtained that admits pseudo-K\"{a}hler structures. In \cite{Smolen-11}, a more complete study of the above classes of six-dimensional pseudo-K\"{a}hlerian nilpotent Lie groups was carried out. Recently, para-complex and para-K\"{a}hler structures have been of great interest. Therefore, the question of invariant para-K\"{a}hler structures on six-dimensional nilpotent Lie groups is natural.
In this article, we will show that 15 classes of non-commutative six-dimensional nilpotent Lie groups admit para-K\"{a}hler structures.

We recall the main concepts and facts that will be used in the work.
An almost para-complex structure on a $2n$-dimensional manifold $M$ is a field $J$ of endomorphisms $J: TM\to TM$ such that $J^2 = Id$, and the ranks of the eigen-distributions $T^\pm:= {\rm ker}(Id\mp J)$ are equal. An almost paracomplex structure $J$ is said to be integrable if the distributions $T^\pm$ are involutive.
In this case, the $J$ is called a para-complex structure.
The Nijenhuis tensor $N$ of an almost paracomplex structure $J$ is defined by the equality $N_J(X,Y) = [X,Y] + [JX,JY] - J[JX,Y] - J[X, JY]$, for any vector fields $X, Y$ on $M$.
As in the complex case, a para-complex structure $J$ is integrable if and only if $N_J = 0$.
A para-K\"{a}hler manifold can be defined as a symplectic manifold $(M, \w)$ with a compatible para-complex structure $J$, i.e. such that $\w(JX,JY)= -\w(X,Y)$.
In this case, the metric $g(X,Y)= \w(X,JY)$ is defined on $M$, which is a pseudo-Riemannian neutral signature. Note that $g(JX,JY) = -g(X,Y)$.
The paper \cite{Aleks} presents a survey of the theory of para-complex structures and considers invariant para-complex and para-K\"{a}hler structures on homogeneous spaces.

Note that the phrase "para-K\"{a}hlerian manifold" is also used in another sense. A. Gray in \cite{Gray} noted that the remarkable geometric and topological properties of K\"{a}hler manifolds are largely due to the fact that the curvature tensor $R$ satisfies the special K\"{a}hler identity: $g(R(X,Y)Z,W) = g(R(X,Y)JZ,JW)$ for any vector fields $X,Y,Z,W$, where $J$ is a complex structure tensor compatible with the Riemannian metric $g$.
However, the class of manifolds with this property is somewhat wider.
In Rizza's 1974 \cite{Rizza} paper, almost Hermitian manifolds satisfying the above K\"{a}hler identity are called para-K\"{a}hler manifolds.
There are many papers devoted to the study of such para-K\"{a}hler manifolds, see, for example, \cite{Banaru} and \cite{Schafer}, where one can also find references to classical and more recent papers.

In this paper, the para-K\"{a}hler manifolds are considered from the point of view of paracomplex geometry. 
For more details on such a para-K\"{a}hlerian geometry, see in review \cite{Aleks}. Note for comparison that such para-K\"{a}hler manifolds satisfy \cite{Aleks} the following identity: $g(R(X,Y)Z,W) =-g(R(X,Y)JZ,JW)$, where $J$ is the tensor of the para-complex structure, compatible with the pseudo-Riemannian metric $g$.

We will consider left-invariant (almost) para-complex structures on the Lie group $G$, which are given by the left-invariant endomorphism field $J:TG \to TG$ of the tangent bundle $TG$.
Since such tensor $J$ is defined by a linear operator on the Lie algebra $\g = T_eG$, we will say that $J$ is an invariant almost para-complex structure on the Lie algebra $\g$.
In this case, the integrability condition for $J$ is also formulated at the Lie algebra level: $N_J(X,Y) = [X,Y] + [JX,JY] - J[JX,Y] - J[X, JY]=0$, for any $X,Y\in \g$. It follows from the integrability condition for $J$ that the eigenspaces $\g^+$ and $\g^-$ of the operator $J$ are subalgebras.
Therefore, the para-complex Lie algebra $\g$ can be represented as a direct sum of two subalgebras:
$$
\g = \g^+ \oplus \g^-.
$$
We also note that for an element $Z$ of the center $\mathcal{Z}(\g)$ of the Lie algebra, the vector $JZ$ may not be central, but it immediately follows from the integrability condition that $ad_{JZ}$ commutes with $J$:
$$
[JZ,JX] =J[JZ,X], \quad   ad_{JZ}\cdot J = J \cdot ad_{JZ}.
$$

A left-invariant symplectic structure $\w$ on a Lie group $G$ is given by a 2-form $\w$ of maximum rank on the Lie algebra $\g$ and satisfying the condition $\w([X,Y],Z) -\w([X,Z],Y) + \w([Y,Z],X) = 0$, $\forall X,Y,Z \in \g$.
In this case, the Lie algebra $\g$ will be called symplectic.

Recall that a subspace $W\subset \g$ is called $\w$-isotropic if and only if $\w(W,W) = 0$ and $W$ is called $\w$-Lagrangian if it is $\w$-isotropic and $\w(W,u) = 0$ implies that $u\in W$.
Subspaces $U,V \subset W$ of a symplectic space $(W,\w)$ is called $\w$-dual if for any vector $u\in U$ there exists a vector $v\in V$ such that $\w(u,v) \ne 0$ and vice versa, $\forall v\in V$, $\exists u \in U$, $\w(u,v) \ne 0$.

A left-invariant para-K\"{a}hler structure on a Lie group is given by a pair $(\w, J)$ consisting of a symplectic form $\w$ on $\g$ and a para-complex structure $J$ on $\g$ compatible with $\w$, i.e., such that $\w(JX,JY) =-\w(X,Y)$.
A compatible pair $(\w, J)$ defines a para-K\"{a}hlerian pseudo-Riemannian metric $g(X,Y) = \w(X,JY)$ on $\g$.
Moreover, the subalgebras $\g^+$ and $\g^-$ are isotropic for the metric $g$ and $\w$-Lagrangian.

The lower central series of the Lie algebra $\g$ is a decreasing sequence of ideals $C^0\g$, $C^{1}\g, \dots $ defined inductively: $C^0\g=\g$ $C^{k+1}\g = [\g, C^{k}\g]$.
A Lie algebra $\g$ is called to be nilpotent if $C^{k}\g = 0$ for some $k$. For a nilpotent Lie algebra, there is also an increasing central sequence of ideals $\{\g_k\}$: $\g_0 = \{0\}\subset \g_1 \subset \g_2 \subset \dots \subset \g_{s-1} \subset \g_p = \g$, where the ideals $\g_k$ are defined inductively by the rule:
$$
\g_k = \{X\in \g\, |\, [X,\g] \subset \g_{k-1}\}, k\ge 1.
$$
In particular, $\g_1 = \mathcal{Z}(\g)$ is the center of the Lie algebra.
Moreover, the first derived ideal $C^1\g$ is included in the ideal $\g_{p-1}$.
For a given almost para-complex structure $J$, the ideals $\{\g_k\}$ defined above are generally not $J$-invariant.
One can define an increasing sequence of $J$-invariant ideals $\ag_k(J)$ as follows:
$$
\ag_0(J) = {0}, \ag_(J) = \{X \in \g\, |\, [X,\g] \subset \ag_{k-1}(J) \mbox{ and } [JX, \g] \subset \ag_{k-1}(J)\}, k \ge 1.
$$

It is clear that $\ag_k(J) \subset \g_k$ for $k \ge 1$.
Obviously, the ideal $\ag_1(J)$ lies at the center $\mathcal{Z}(\g) = \g_1$ of the Lie algebra $\g$.

\begin{definition}
A left-invariant almost para-complex structure $J$ is called nilpotent if for a sequence of ideals $\{\ag_k(J)\}$ there exists a number $s$ such that $\ag_s(J) = \g$.
\end{definition}

Let $\nabla$ be the Levi-Civita connection corresponding to the left-invariant pseudo-Riemannian metric $g$.
It is determined from the six-term formula \cite{KN}, which for left-invariant vector fields $X,Y,Z$ on a Lie group takes the form: $2g(\nabla_X Y,Z) = g([X,Y],Z) + g([Z,X],Y) + g([Z,Y],X)$.
Recall that the curvature tensor $R(X,Y)$ is defined by the formula $R(X,Y) = [\nabla_X ,\nabla_Y] - \nabla_{[X,Y]}$, and the Ricci tensor $Ric$ is the convolution of the curvature tenor with respect to the first lower and upper indices, $Ric_{jk}=R_{ijk}^i$.
The metric $g$ is called Ricci-flat if $Ric =0$.

\section{Left-invariant symplectic and para-K\"{a}hler structures on Lie algebras} \label{Sec-1}
Let $\g$ be a Lie algebra with a symplectic form $\w$, an almost paracomplex structure $J$ compatible with $\w$, and a pseudo-Riemannian metric $g(X,Y) = \w(X,JY)$.
Let us present some simple facts about the first derived ideal $C^1(\g)$, the center $\mathcal{Z}(\g)$ of the Lie algebra $\g$, and the ideals $\{\g_k\}$ and $\{\ag_k(J)\}$.

\begin{proposition}\label{Prop-1}
$\w(C^1(\g), \mathcal{Z}(\g)) = 0$ for any symplectic form $\w$ on $\g$.
\end{proposition}

\begin{proof}
Immediately follows from the formula $d\w(X,Y,Z) = \w([X,Y],Z)- \w([X,Z],Y) + \w([Y,Z],X) = \w([X,Y],Z) = 0$, $\forall X,Y,Z \in \g$.
\end{proof}

\begin{proposition}\label{Prop-2}
For any symplectic form $\w$ on $\g$ and an almost paracomplex structure $J$ compatible with $\w$, we have $\w(C^1\g\oplus J(C^1\g),\ag_1(J)) = 0$.
\end{proposition}

\begin{proof}
Since $\ag_1(J)$ belongs to center $\mathcal{Z}(\g)$, then $\w(C^1\g,\ag_1(J)) = 0$. The equality $\w(J(C^1\g),\ag_1(J)) = 0$ follows from the $J$-invariance of $\ag_1(J)$ and formulas $\w(JX,JY) = -\w(X,Y)$.
\end{proof}

\begin{corollary}\label{Cor-1}
For any almost para-K\"{a}hler structure $(\g, \w, g, J)$, the ideal $\ag_1(J)\subset \g_1$ is orthogonal to the subspace $C^1\g\oplus J(C^1\g)$: $g(C^1\g\oplus J(C^1\g),\ag_1(J)) = 0$.
\end{corollary}

Let $\g$ be a nilpotent Lie algebra and let $\g_0 = \{0\}\subset \g_1 \subset \g_2 \subset \dots \subset \g_{s-1} \subset \g_p = \g$ be the increasing central sequence of ideals.

\begin{proposition}\label{Prop-3}
For any nilpotent almost para-K\"{a}hler structure $J$, the ideal $\ag_{p-1}(J)$ contains $C^1\g\oplus J(C^1\g)$.
\end{proposition}

\begin{proof}
Since $\g_p =\{X\in \g\, |\, [X,\g] \subset \ag_{p-1}(J)$ and $[JX,\g] \subset \ag_{p-1}(J)\} = \g$, then $C^1\g \subset \ag_{p-1}(J)$. Therefore $J(C^1\g) \subset J(\ag_{p-1}(J)) = \ag_{p-1}(J)$.
\end{proof}

\begin{proposition}\label{Prop-4}
For any left-invariant (pseudo) Riemannian structure $g$ on the nilpotent Lie algebra $\g$, the following properties hold:
\begin{enumerate}
  \item If the vector $X$ belongs to center $\mathcal{Z}(\g)$ of the Lie algebra, then $\nabla_X Y = \nabla_Y X$, $\forall Y \in \g$.
  \item If the vectors $X$ and $Y$ belongs to center of the Lie algebra, then $\nabla_X Y = 0$.
\end{enumerate}
\end{proposition}

\begin{proof}
It follows from the formula $2g(\nabla_X Y,Z) = g([X,Y],Z) + g([Z,X],Y) + g(X,[Z,Y])$ for the covariant derivative $\nabla$ on the Lie group.
Indeed, if the vector $X$ lies at the center of the Lie algebra, then $2g(\nabla_X Y,Z) = g(X,[Z,Y])$ and $2g(\nabla_Y X,Z) = g([Y,X],Z) + g([Z,Y],X) + g(Y,[Z,X]) = g([Z,Y],X)$, $\forall Z\in \g$.
\end{proof}

\begin{corollary}\label{Cor-2}
If the vectors $X,Y$, $Z$ belongs to the center of the Lie algebra, then $R(X,Y)Z = 0$.
\end{corollary}

Let $J$ be a nilpotent almost para-complex structure on a nilpotent Lie algebra $\g$ and let $\ag_0 = \{0\} \subset \ag_1(J) \subset \dots \subset \ag_{s-1}(J) \subset \ag_s(J)$ be the corresponding sequence of ideals. Note that although the ideals $\ag_k(J)$ are $J$-invariant, they need not be even-dimensional.

\begin{proposition}\label{Prop-5}
If the vector $X$ belongs to the ideal $\ag_1(J)\subset \mathcal{Z}(\g)$ of the Lie algebra, then $\nabla_X Y = \nabla_Y X = 0$, $\forall Y \in \g$.
\end{proposition}

\begin{proof}
Let $X\in \ag_1(J) \subset \g_1 = \mathcal{Z}(\g)$ and $Z,Y\in \g$. Then the covariant derivative formula and Corollary \ref{Cor-1} imply that $2g(\nabla_X Y,Z) = g(X,[Z,Y]) = 0$.
\end{proof}

\begin{corollary}\label{Cor-3}
If the vector $X$ belongs to the ideal $\ag_1(J) \subset \g_1 = \mathcal{Z}(\g)$ of the Lie algebra, then $R(X,Y)Z = R(Z,Y)X = 0$, $\forall Y,Z \in \g$.
\end{corollary}

\begin{proof}
It follows from $\nabla_X Y = \nabla_Y X = 0$, $\forall Y \in \g$ and the formulas $R(X,Y)Z = \nabla_X(\nabla_Y Z) -\nabla_Y(\nabla_X Z) -\nabla_{[X,Y]}Z$ and $R(Z,Y)X = \nabla_Z(\nabla_Y X)- \nabla_Y(\nabla_Z X)- \nabla_{[Z,Y]}X$.
\end{proof}

\section{Left-invariant para-K\"{a}hler structures on six-dimensional nilpotent Lie algebras} \label{Sec-3}

A classification list of six-dimensional symplectic nilpotent Lie algebras is presented in the article Goze M., Khakimdjanov Y., Medina A. \cite{Goze-Khakim}.
Many Lie algebras in this list have an increasing sequence of ideals $\mathcal{Z}(\g) = \g_1 \subset \g_2 \subset \g_3 = \g$ of dimensions 2, 4, and 6 (we say that such a Lie algebra is of type (2,4,6)).
We choose the complement $A$ to $\g_2$ and the complement $B$ to $\mathcal{Z}(\g)$ in $\g_2$.
Then such a Lie algebra can be represented as a direct sum of two-dimensional subspaces
$$
\g = A \oplus B \oplus \mathcal{Z}(\g).
$$
From the definition of ideals $\g_3 = \{X\in \g \,|\, [X,\g] \subset \g_2\} = \g$ and $\g_2 = \{X \in \g \,|\, [X,\g] \subset \mathcal{Z}(\g)\}$ it immediately follows that the additional subspaces $A$ and $B$ have the following properties:
$$
[A, A] \subset \g_2 = B \oplus \mathcal{Z}(\g),\qquad  [A,B] \subset \mathcal{Z}(\g).
$$
For example, for the algebra $\mathfrak{G}_{21}$ of the list in \cite{Goze-Khakim} with Lie brackets $[e_1,e_2] = e_4$, $[e_1,e_4] = e_6$, $[e_2,e_3] = e_6$ we have: $\mathcal{Z}(\g) = \mathbb{R}\{e_5,e_6\}$, $\g_2 = \mathbb{R}\{e_3, e_4, e_5, e_6\}$, $\g_3 = \g$. Then one can choose as subspaces $A$ and $B$: $B = \mathbb{R}\{e_3, e_4\}$ and $A = \mathbb{R}\{e_1, e_2\}$.

The symplectic structures of the classification list in \cite{Goze-Khakim} show that for algebras of type (2,4,6) additional subspaces $A$ and $B$ can be chosen so that the symplectic form $\w$ on the subspace $B$ is non-degenerate, and the subspaces $A$ and $\mathcal{Z}(\g)$ are $\w$-isotropic and $\w$-dual.
For Lie algebras of other types, instead of $\mathcal{Z}(\g)$ it is necessary to choose a two-dimensional subspace $C\subset \mathcal{Z}(\g)$.

\begin{theorem}\label{Th-1}
Let a six-dimensional symplectic Lie algebra $(\g, \w)$ have a decomposition in the form of a direct sum of two-dimensional subspaces
$$
\g = A \oplus B \oplus C,
$$
where $C\subset \mathcal{Z}(\g)$, $[A, A] \subset B \oplus C$ and $[A,B] \subset C$.
Assume that $B \oplus C$ is an Abelian subalgebra, the subspaces $A$ and $C$ are $\w$-isotropic and $\w$-dual, on subspace $B$ the form $\w$ is non-degenerate and $\w(B \oplus C, C) = 0$.
Then for any nilpotent almost para-complex structure $J$ compatible with $\w$ and for which the subspaces $B$ and $C$ are $J$-invariant, the Levi-Civita connection $\nabla$ of the corresponding pseudo-Riemannian metric $g_J$ has property location:
\begin{enumerate}
  \item $\nabla_XY \in B \oplus C, \, \forall X,Y \in \g$,
  \item $\nabla_XY, \nabla_YX \in C, \, \forall X \in \g, \, \forall Y \in B \oplus C$,
  \item $\nabla_XY = 0, \, \forall X,Y \in B \oplus C$.
\end{enumerate}
\end{theorem}

\begin{proof}
Property 1. Let $X,Y \in \g$. Then $[X,Y]\subset B \oplus C$. Suppose that $\nabla_XY$ does not belongs to space $B\oplus C$, i.e., has a nonzero component from $A$.
Then there exists a vector $Z \in C \subset \mathcal{Z}(\g)$ such that $\w(\nabla_XY,JZ)\ne 0$.
At the same time, $2\w(\nabla_XY,JZ) = 2g(\nabla_XY,Z) = g([X,Y],Z) + g([Z,X],Y) + g([Z,Y],X) = g([X,Y],Z) = \w([X,Y],JZ) = 0$.
The last equality follows from the property $\w(B\oplus C, C)= 0$.

Property 2. Let now $X \in\g$ and $Y\in B\oplus C$. Then $\nabla_XY$ belongs to space $B\oplus C$.
Assume that $\nabla_XY$ does not belongs to space $C$, i.e., has a nonzero component from $B$.
Then the condition that the form $\w$ is non-degenerate on $B$ and the equality $\w(B\oplus C, C)= 0$ implies that there exists a vector $Z\in B$ such that $JZ\in B$ and $\w(\nabla_XY,JZ) \ne 0$.
At the same time, taking into account the commutativity of $B\oplus C$ and the inclusion $[A,B]\subset C \subset  \ag_1(J) \subset  \mathcal{Z}(\g)$, we have: $2\w(\nabla_XY,JZ)= 2g(\nabla_XY,Z)= g([X,Y],Z) + g([Z,X],Y) + g([Z,Y],X)= g([X,Y],Z)+ g([Z,X],Y) = \w([X,Y],JZ) + \w([Z,X],JY) = 0$.
The last equality follows from the fact that $JY,JZ \in B\oplus C$, $[X,Y], [Z,X] \in C$ and $\w(B\oplus C,C)= 0$.
Thus, $\nabla_XY \in C$.
The inclusion $\nabla_YX \in C$ follows from the formula $\nabla_XY -\nabla_YX = [X,Y]$.

Property 3. Let $X,Y\in B\oplus C$.
Then, for any $Z\in \g$, the following holds in exactly the same way: $2g(\nabla_XY,Z) = g([X,Y],Z) + g([Z,X],Y ) + g([Z,Y],X) = \w([Z,X],JY )+\w([Z,Y],JX) = 0$.
\end{proof}

\begin{corollary}\label{Cor-4}
Under the assumptions of Theorem 1, if the vector $X$ belongs to space $B\oplus C$, then the following equalities hold: $R(X,Y)Z = R(Z,Y)X = 0$ for any $Y,Z\in \g$.
\end{corollary}

\begin{proof}
Let $X\in B\oplus C$ and let $Y,Z\in \g$.
In the formula $R(X,Y)Z = \nabla_X \nabla_YZ -\nabla_Y \nabla_XZ - \nabla_{[X,Y]}Z$ we have $[X,Y]\in C\subset \ag_1(J)$, so by Proposition \ref{Prop-5}, $\nabla_{[X,Y]}Z=0$.
Further, $\nabla_YZ\in B\oplus C$, so $\nabla_X \nabla_YZ = 0$ by property 3 of the theorem. By property 2 of the theorem, we have $\nabla_XZ \in C\subset \ag_1(J) \subset \mathcal{Z}(\g)$. Then $\nabla_Y\nabla_XZ = 0$ by Proposition \ref{Prop-5}.

Let $X\in B\oplus C$ and $R(Z,Y)X = \nabla_Z\nabla_YX -\nabla_Y\nabla_ZX -\nabla_{[Z,Y]}X$.
We have $[Z,Y]\in B\oplus C$, so $\nabla_{[Z,Y]}X = 0$.
Since $\nabla_YX \in C\subset  \ag_1(J) \subset \mathcal{Z}(\g)$, then $\nabla_Z\nabla_YX = 0$ by Proposition \ref{Prop-5}. Similarly, $\nabla_Y\nabla_ZX =0$. Therefore, $R(Z,Y)X = 0$.
\end{proof}

We choose a basis $\{e_1, e_2, e_3, e_4, e_5,e_6\}$ of the Lie algebra $\g$ adapted to the decomposition $\g = A\oplus B \oplus C$, i.e., such that $A = \mathbb{R}\{e_1, e_2\}$, $B = \mathbb{R}\{e_3, e_4\}$ and $C = \mathbb{R}\{e_5,e_6\}$.

\begin{corollary}\label{Cor-5}
Under the assumptions of Theorem 1, for any $X,Y,Z \in \g$, the inclusion $R(X,Y)Z \in C \subset \mathcal{Z}(\g)$ holds.
In the basis adapted to the expansion $\g = A\oplus B \oplus C$, the curvature tensor can have, up to symmetries, only four non-zero components $R_{1,2,1}^5$, $R_{1,2,1}^6$, $R_{1,2,2}^5$, $R_{1,2,2}^6$. In particular, the Ricci tensor is zero.
\end{corollary}

\begin{proof}
Let $X,Y,Z \in \g$.
In the formula $R(X,Y)Z = \nabla_X \nabla_YZ -\nabla_Y \nabla_XZ - \nabla_{[X,Y]}Z$ we have $[X,Y] \in B \oplus C$, so $\nabla_{[X,Y]}Z \in C$. Further, $\nabla_YZ \in B\oplus C$, so $\nabla_X \nabla_YZ \in C$.
Similarly, $\nabla_Y\nabla_ZX \in C$.
Thus, $R(X,Y)Z\in C$.
The assertion about nonzero components follows from Corollary \ref{Cor-4}.
\end{proof}

Consider the question of which of the Lie algebras in the classification list in \cite{Goze-Khakim} admit compatible para-complex structures $(\w, J)$.
The results are presented in the table of Theorem 2 below.
Each Lie algebra in the table has its number from the list of symplectic Lie algebras in \cite{Goze-Khakim}.
For each symplectic Lie algebra $(\g, \w)$ of this table there exist multiparameter families of para-complex structures $J$ compatible with $\w$.
The table of Theorem 2 shows one of them, the simplest $J$, which is presented in block form in the basis $\{e_1, e_2, e_3, e_4, e_5, e_6\}$ of the Lie algebra and, in accordance with the expansion $\g = \mathbb{R}\{e_1,e_2\}\oplus \mathbb{R}\{e_3,e_4\}\oplus  \mathbb{R}\{e_5,e_6\}$.
The symbol $R$ denotes the Riemann tensor. The dual basis is denoted by the symbols $\{e^1,\dots , e^6\}$.

\begin{theorem}\label{Th-2}
A six-dimensional nilpotent noncommutative Lie algebra admits a para-K\"{a}hler structure $(J, \w)$ if and only if it is symplecto-isomorphic to one of the algebras in the table below. The admissible para-complex structures $J$ are nilpotent, and the corresponding pseudo-Riemannian metrics are Ricci-flat.
\end{theorem}

{\small
\begin{longtable}[H]{|l|l|l|}
\hline
$\g$ & Lie brackets & Para-K\"{a}hler structure \\

\hline
$\G_6$ &
\begin{tabular}{l}
  $[e_1,e_2]= e_3$,\\
  $[e_1,e_3] = e_4$,\\
  $[e_1,e_4] = e_5$,\\
  $[e_2,e_3] = e_6$,
\end{tabular} &
\begin{tabular}{l}
  $\omega =e^1\wedge e^6 + e^2\wedge e^4 + e^2\wedge e^5 -e^3\wedge e^4$,\\
$J=\left(
     \begin{array}{cc}
       1 & 0 \\
       0 & -1 \\
     \end{array}
   \right) \times
\left(
     \begin{array}{cc}
       -1 & 0 \\
       0 & 1 \\
     \end{array}
   \right) \times
\left(
     \begin{array}{cc}
       1 & 0 \\
       0 & -1 \\
     \end{array}
   \right)$, \\
 $R\ne 0$
 \end{tabular}
\\

\hline
$\G_{10}$ &
\begin{tabular}{l}
  $[e_1,e_2]= e_4$,\\
  $[e_1,e_4] = e_5$,\\
  $[e_1,e_3] = e_6$,\\
  $[e_2,e_4] = e_6$,
\end{tabular} &
\begin{tabular}{l}
  $\omega =e^1\wedge e^6 + e^2\wedge e^5 -e^3\wedge e^4 -e^2\wedge e^6$,\\
$J=\left(\begin{array}{cc}
       1 & a \\
       0 & -1 \\
     \end{array}
   \right) \times
\left(\begin{array}{cc}
       -1 & 0 \\
       -\frac{a(a+2)}{2(a+1)} & 1 \\
     \end{array}
   \right) \times
\left(\begin{array}{cc}
       1 & -a-2 \\
       0 & -1 \\
     \end{array}
   \right)$,  $R\ne 0$ when $a\ne 0$
 \end{tabular}
\\

\hline
$\G_{11}$ &
\begin{tabular}{l}
  $[e_1,e_2]= e_4$,\\
  $[e_1,e_4] = e_5$,\\
  $[e_2,e_3] = e_6$,\\
  $[e_2,e_4] = e_6$,
\end{tabular} &
\begin{tabular}{l}
  $\omega=e^1\wedge e^6 + e^2\wedge e^5 - e^3\wedge e^4 +\lambda e^2\wedge e^6$,\\
$J=\left(\begin{array}{cc}
       -1 & -2\lambda  \\
       0 & 1 \\
     \end{array}
   \right) \times
\left(\begin{array}{cc}
       1 & a \\
       0 & -1 \\
     \end{array}
   \right) \times
\left(\begin{array}{cc}
       -1 & 0 \\
       0 & 1 \\
     \end{array}
   \right)$,  $R\ne 0$ when $a\ne 1$
 \end{tabular}
\\

\hline
$\G_{12}$ &
\begin{tabular}{l}
  $[e_1,e_2]= e_4$,\\
  $[e_1,e_4] = e_5$,\\
  $[e_2,e_3] = -e_5$,\\
  $[e_1,e_3] = e_6$, \\
  $[e_2,e_4] = e_6$
\end{tabular} &
\begin{tabular}{l}
  $\omega=-e^1\wedge e^5 + \lambda e^2\wedge e^6 + (\lambda-1) e^3\wedge e^4$,\\
$J=\left(\begin{array}{cc}
       0 & \lambda a \\
       \frac{1}{\lambda a} & 0 \\
     \end{array}
   \right) \times
\left(\begin{array}{cc}
       0 & \frac{\lambda a^2 -1}{a(\lambda +1)} \\
       \frac{a(\lambda +1)}{\lambda a^2 -1} & 0 \\
     \end{array}
   \right) \times
\left(\begin{array}{cc}
       0 & \frac{1}{a} \\
       a & 0 \\
     \end{array}
   \right)$,  $R\ne 0$
 \end{tabular}
\\

\hline
$\G_{13}$ &
\begin{tabular}{l}
  $[e_1,e_2]= e_4$,\\
  $[e_1,e_3] = e_5$,\\
  $[e_1,e_4] = e_6$,\\
  $[e_2,e_3] = e_6$
\end{tabular} &
\begin{tabular}{l}
  $\omega_1= e^1\wedge e^6 + \lambda e^2\wedge e^5 +(\lambda-1) e^3\wedge e^4$,    $\lambda \ne 0, \lambda \ne 1$\\
$J_1=\left(\begin{array}{cc}
       1 & 0 \\
       a & -1 \\
     \end{array}
   \right) \times
\left(\begin{array}{cc}
       1 & 0 \\
       -a(\lambda +1) & -1 \\
     \end{array}
   \right) \times
\left(\begin{array}{cc}
       1 & 0 \\
       -a\lambda & -1 \\
     \end{array} \right), \ R = 0$ \\
  $\omega_2= e^1\wedge e^6 + e^2\wedge e^4 +\frac 12 e^2\wedge e^5+\frac 12 e^3\wedge e^4$,  $\lambda \ne 0, \lambda \ne 1$ \\
$J_2=\left(\begin{array}{ccc}
       1 & 0 & 0\\
       0 & -1 & 0\\
       0 & 0 & 1 \\
     \end{array}
   \right) \times
\left(\begin{array}{ccc}
       -1 & 0 & 0\\
       4 & 1 & 0\\
       0 & 0 & -1 \\
     \end{array}
   \right), \ R \ne 0 $
\end{tabular}
\\

\hline
$\G_{14}$ &
\begin{tabular}{l}
  $[e_1,e_2]= e_4$,\\
  $[e_1,e_3] = e_5$,\\
  $[e_1,e_4] = e_6$
\end{tabular} &
\begin{tabular}{l}
  $\omega=e^1\wedge e^6 +  e^2\wedge e^5 + e^3\wedge e^4$ \\
$J=\left(\begin{array}{cc}
       a & b \\
       \frac{1 -a^2}{b} & -a \\
     \end{array}
   \right) \times
\left(\begin{array}{cc}
       a & -b \\
       \frac{a^2-1}{b} & -a \\
     \end{array}
   \right) \times
\left(\begin{array}{cc}
       a & -b \\
       \frac{a^2-1}{b} & -a \\
     \end{array} \right), \ R = 0$
 \end{tabular}
\\

\hline
$\G_{15}$ &
\begin{tabular}{l}
  $[e_1,e_2]= e_4$,\\
  $[e_2,e_3] = e_5$,\\
  $[e_1,e_4] = e_6$
\end{tabular} &
\begin{tabular}{l}
  $\omega=-e^1\wedge e^5 +  e^1\wedge e^6 + e^2\wedge e^5$ \\
$J=\left(\begin{array}{cc}
       1 & 0 \\
       a & -1 \\
     \end{array}
   \right) \times
\left(\begin{array}{cc}
       1 & 0 \\
       \frac{a(a-2)}{2} & -1 \\
     \end{array}
   \right) \times
\left(\begin{array}{cc}
       1 & 0 \\
       2-a & -1 \\
     \end{array} \right), \ R \ne 0$
 \end{tabular}
\\

\hline
$\G_{16}$ &
\begin{tabular}{l}
  $[e_1,e_2]= e_5$,\\
  $[e_1,e_4] = -e_5$,\\
  $[e_1,e_3] = e_6$, \\
  $[e_2,e_4] = e_6$
\end{tabular} &
\begin{tabular}{l}
  $\omega =e^1\wedge e^6 + e^2\wedge e^3 -e^4\wedge e^5$ \\
$J=e_1\otimes e^1 - e_2\otimes e^2 + e_3\otimes e^3 - e_4\otimes e^4 + e_5\otimes e^5 - e_6\otimes e^6+ $ \\
$ +\frac{4-a^2}{2a} e_2\otimes e^3 + a\,e_4\otimes e^1 +  a\, e_6\otimes e^5,  \ R \ne 0$
 \end{tabular}
\\

\hline
$\G_{17}$ &
\begin{tabular}{l}
  $[e_1,e_3]= e_5$,\\
  $[e_1,e_4] = e_6$,\\
  $[e_2,e_3] = e_6$
\end{tabular} &
\begin{tabular}{l}
  $\omega= e^1\wedge e^6 + e^2\wedge e^5 + e^3\wedge e^4$ \\
$J=\left(\begin{array}{cc}
       1 & 0 \\
      \frac{a-1}{b} & -1 \\
     \end{array}
   \right) \times
\left(\begin{array}{cc}
       a & b \\
       \frac{1-a^2}{b} & -a \\
     \end{array}
   \right) \times
\left(\begin{array}{cc}
       1 & 0 \\
       \frac{1-a}{b} & -1 \\
     \end{array} \right), \ R = 0$
 \end{tabular}
\\

\hline
$\G_{18}$ &
\begin{tabular}{l}
  $[e_1,e_2]= e_4$,\\
  $[e_1,e_3] = e_5$,\\
  $[e_2,e_3] = e_6$
\end{tabular} &
\begin{tabular}{l}
  $\omega_1= e^1\wedge e^6 + \lambda e^2\wedge e^5 + (\lambda-1) e^3\wedge e^4$, $\lambda\ne 0$ and $\ne 1$ \\
$J_1=\left(\begin{array}{cc}
       1 & 0 \\
      0 & 1 \\
     \end{array} \right) \times
\left(\begin{array}{cc}
       -1 & 0 \\
       a & 1 \\
     \end{array} \right) \times
\left(\begin{array}{cc}
       -1 & 0 \\
       0 & -1 \\
     \end{array} \right) $, $R=0$ \\
 $\omega_2= e^1\wedge e^5 + \lambda e^1\wedge e^6 -\lambda e^2\wedge e^5 + e^2\wedge e^6 -2\lambda e^3\wedge e^4$, $\lambda\ne 0$ \\
$J_2=J_1$,  $R = 0$,\\
 $\omega_3= -e^1\wedge e^6 + e^2\wedge e^5 +2 e^3\wedge e^4 + e^3\wedge e^5$, \\
$J_3=\left(\begin{array}{ccc}
      1 & 0 & 0 \\
      0 & 1 & 2 \\
      0 & 0 & -1 \\
     \end{array} \right) \times
\left(\begin{array}{ccc}
      1 & 0 & 0 \\
      0 & -1 & 0 \\
      0 & 0 & -1 \\
     \end{array} \right)$, $R=0$ \\
 \end{tabular}
\\

\hline
$\G_{19}$ &
\begin{tabular}{l}
  $[e_1,e_2]= e_4$,\\
  $[e_1,e_4] = e_5$,\\
  $[e_1,e_5] = e_6$
\end{tabular} &
\begin{tabular}{l}
  $\omega= e^1\wedge e^3 + e^2\wedge e^6 - e^4\wedge e^5$ \\
$J=\left(\begin{array}{cc}
       1 & 0 \\
      0 & -1 \\
     \end{array}
   \right) \times
\left(\begin{array}{cc}
       -1 & 0 \\
       0 & -1 \\
     \end{array}
   \right) \times
\left(\begin{array}{cc}
       1 & 0 \\
       0 & 1 \\
     \end{array} \right), \ R \ne 0$
 \end{tabular}
\\

\hline
$\G_{21}$ &
\begin{tabular}{l}
  $[e_1,e_2]= e_4$,\\
  $[e_1,e_4] = e_6$,\\
  $[e_2,e_3] = e_6$
\end{tabular} &
\begin{tabular}{l}
  $\omega= e^1\wedge e^6 + e^2\wedge e^5 - e^3\wedge e^4$ \\
$J=\left(\begin{array}{cc}
       a & b \\
      \frac{1-a^2}{b} & -a \\
     \end{array}
   \right) \times
\left(\begin{array}{cc}
       a & -b \\
       \frac{a^2-1}{b} & -a \\
     \end{array}
   \right) \times
\left(\begin{array}{cc}
       a & -b \\
       \frac{a^2-1}{b} & -a \\
     \end{array} \right), \ R \ne 0$
 \end{tabular}
\\

\hline
$\G_{23}$ &
\begin{tabular}{l}
  $[e_1,e_2]= e_5$,\\
  $[e_1,e_3] = e_6$
\end{tabular} &
\begin{tabular}{l}
  $\omega_1= e^1\wedge e^6 + e^2\wedge e^5 + e^3\wedge e^4$  \\
$J_1=\left(\begin{array}{cc}
      -1 & 0 \\
      a & 1 \\
     \end{array} \right) \times
\left(\begin{array}{cc}
       1 & 0 \\
       b & -1 \\
     \end{array} \right) \times
\left(\begin{array}{cc}
       -1 & 0 \\
       -a & 1 \\
     \end{array} \right) $, $R=0$ \\
 $\omega_2= e^1\wedge e^4 + e^2\wedge e^6 + e^3\wedge e^5$ \\
$J_2=\left(\begin{array}{cc}
      -1 & 0 \\
      0 & 1 \\
     \end{array} \right) \times
\left(\begin{array}{cc}
       1 & 0 \\
       0 & -1 \\
     \end{array} \right) \times
\left(\begin{array}{cc}
       -1 & 0 \\
       0 & 1 \\
     \end{array} \right) $, $R=0$ \\
$\omega_3= e^1\wedge e^4 + e^2\wedge e^5 - e^3\wedge e^6$, \\
$J_3=\left(\begin{array}{cc}
      1 & 0 \\
      0 & -1 \\
     \end{array} \right) \times
\left(\begin{array}{cc}
       -1 & 0 \\
       0 & -1 \\
     \end{array} \right) \times
\left(\begin{array}{cc}
       1 & 0 \\
       0 & 1 \\
     \end{array} \right)$, $R=0$
 \end{tabular}
\\

\hline
$\G_{24}$ &
\begin{tabular}{l}
  $[e_2,e_3]= e_5$,\\
  $[e_1,e_4] = e_6$
\end{tabular} &
\begin{tabular}{l}
  $\omega= e^1\wedge e^6 + e^2\wedge e^5 + e^3\wedge e^4$  \\
$J=\left(\begin{array}{cc}
      1 & a \\
      0 & -1 \\
     \end{array} \right) \times
\left(\begin{array}{cc}
       1 & a^2/2 \\
       0 & -1 \\
     \end{array} \right) \times
\left(\begin{array}{cc}
       1 & -a \\
       0 & -1 \\
     \end{array} \right)$, $R\ne 0$
 \end{tabular}
\\

\hline
$\G_{25}$ &
\begin{tabular}{l}
  $[e_1,e_2]= e_6$
\end{tabular} &
\begin{tabular}{l}
  $\omega= e^1\wedge e^6 + e^2\wedge e^5 + e^3\wedge e^4$  \\
$J=\left(\begin{array}{cc}
      1 & 0 \\
      a & -1 \\
     \end{array} \right) \times
\left(\begin{array}{cc}
       1 & 0 \\
       b & -1 \\
     \end{array} \right) \times
\left(\begin{array}{cc}
       1 & 0 \\
       -a & -1 \\
     \end{array} \right)$, $R=0$
 \end{tabular}
\\

\hline

\caption{Para-K\"{a}hler six-dimensional nilpotent Lie algebras}
\end{longtable}
}

\begin{proof}
A left-invariant para-K\"{a}hler structure is a pair $(\w, J)$ consisting of a symplectic form $\w$ and an integrable almost para-complex structure $J$ compatible with $\w$ on the Lie algebra $\g$.
Since the symplectic form is given in \cite{Goze-Khakim}, the operator $J$ must have the following properties: $J^2 = Id$, $\w(JX,JY) = -\w(X,Y)$ and $[X,Y] + [JX,JY] - J[JX,Y] -J[X,JY] = 0$.
We write the compatible condition $\w(JX,JY) = -\w(X,Y)$ as $\w(JX,Y) + \w(X,JY) = 0$.
Let the symplectic form $\w$ and operator $J$ have matrices $\w_{ij}$ and $J_j^i$, $J=J_j^i e_i\otimes e^j$ in the basis $\{e_1,\dots , e_6\}$ of the Lie algebra.
Then the system of equations for finding the para-K\"{a}hler structure $(\w, J)$ consists of the following equations for the variables :
$$
\left\{
\begin{tabular}{l}
$J_k^i\, J^k_j = \delta_j^i,$  \\
$\w_{kj}J^k_i+\w_{ik}J^k_j=0$, \\
$J_i^l J_j^m C_{lm}^k-J_i^l J_m^k C_{lj}^m-J_j^l J_m^k C_{il}^m+C_{ij}^k=0$,
 \end{tabular}
\right.
$$
where $\delta_j^i$ is the identity matrix, $C_{ij}^k$ are the structure constants of the Lie algebra, and the indices vary from 1 to 6.
For the Lie algebras included in the table of Theorem 2, it is easy to see that the reduced almost para-complex structures $J$ are compatible with the symplectic forms $\w$.
The integrability of $J$ immediately follows from the fact that the proper subspaces $\g^+$ and $\g^-$ are subalgebras.
For all Lie algebras in this list, with the exception of $\G_6$, the conditions of Theorem 1 are satisfied. Therefore, the corresponding para-K\"{a}hler metrics for them are Ricci flat.

Let us consider in more detail the Lie algebra $\G_6$ with symplectic form $\omega =e^1\wedge e^6 + e^2\wedge e^4 + e^2\wedge e^5 -e^3\wedge e^4$. For the almost para-complex structure
$J=\left(     \begin{array}{cc}
       1 & 0 \\
       0 & -1 \\
     \end{array} \right) \times
\left(\begin{array}{cc}
       -1 & 0 \\
       0 & 1 \\
     \end{array} \right) \times
\left(\begin{array}{cc}
       1 & 0 \\
       0 & -1 \\
     \end{array} \right)$
we have the following sequence of ideals: $\ag_1(J) = \mathcal{Z}(\g) = \mathbb{R}\{e_5, e_6\}$, $\ag_2(J) = \mathbb{R}\{e_4, e_5, e_6\}$, $\ag_3(J) = \mathbb{R}\{e_3, e_4, e_5, e_6\}$, $\ag_4(J) = \g$. Therefore the $J$ is nilpotent.
The eigenspaces $\g^+$ and $\g^-$ are formed by the vectors $\{e_1, e_4, e_5\}$ and $\{e_2, e_3, e_6\}$. It is easy to see that they are subalgebras.
Therefore, the almost para-complex structure $J$ is integrable.
The compatible condition $\w(JX,JY) =-\w(X,Y)$ is obviously satisfied.
Using Maple system, the curvature tensor is easily calculated. It has the following non-zero components: $R_{1,2,1}^4=1$, $R_{1,2,2}^6=1$, $R_{1,2,3}^6=-1$, $R_{1,3,1}^5=-1$, $R_{1, 3,2}^6=-1$.
Therefore, the para-K\"{a}hler metric is also Ricci-flat.
Note that the most general para-K\"{a}hler structure for a given Lie algebra $\G_6$ depends on five parameters and has the form:
$$
J=
\left(
  \begin{array}{cccccc}
    1 & 0 & 0 & 0 & 0 & 0 \\
    0 & -1 & 0 & 0 & 0 & 0 \\
    J_1^3 & 0 & -1 & 0 & 0 & 0 \\
    -\frac{J_1^3 J_3^4}{2} & J_2^4 & J_3^4 & 1 & 0 & 0 \\
    \frac{J_1^3 (J_2^4+J_3^4}{2} & J_2^5 & -J_2^4-J_3^4 & 0 & 1 & 0 \\
    J_1^6 & \frac{J_1^3 J_2^4}{2} & \frac{J_1^3 J_3^4}{2} & J_1^3 & 0 & -1 \\
  \end{array}
\right), \
g_J=
\left(
  \begin{array}{cccccc}
 J_1^6 & \frac{J_1^3 J_2^4}{2}& \frac{J_1^3 J_3^4}{2}&J_1^3&0&-1 \\
    \frac{J_1^3 J_2^4}{2}&J_2^4+J_2^5&-J_2^4&1&1&0 \\
    \frac{J_1^3 J_3^4}{2}&-J_2^4&-J_3^4&-1&0&0 \\
    J_1^3&1&-1&0&0&0 \\
    0&1&0&0&0&0 \\
    -1&0&0&0&0&0 \\
  \end{array}
\right).
$$

For all Lie algebras in the list of work \cite{Goze-Khakim}, which are not included in the table of Theorem 2, the system of equations for the para-K\"{a}hler structure has no real solutions.
Consider, for example, the Lie algebra $\G_1$ with Lie brackets
$[e_1,e_2] = e_3$, $[e_1,e_3] = e_4$, $[e_1,e_4] = e_5$, $[e_1,e_5] = e_6$, $[e_2,e_3] = e_5$, $[e_2,e_4] = e_6$, and structure constants
$C_{12}^3=1$, $C_{21}^3=-1$, $C_{13}^4=1$, $C_{31}^4=-1$, $C_{14}^5=1$, $C_{41}^5=-1$, $C_{15}^6=1$, $C_{51}^6=-1$, $C_{23}^5=1$, $C_{32}^5=-1$, $C_{24}^6=1$, $C_{42}^6=-1$.
The symplectic form is $\w = e_1\wedge e_6 + (1-\lambda)e_2\wedge e_5 + \lambda e_3\wedge e_4$, where $\lambda \neq 0$ and $\lambda \neq 1$.
Consider a matrix of almost paracomplex structure $J=(J_j^i)$, whose elements are denoted by symbols $\psi_{ij}$, $J_j^i=\psi_{ij}$.
From the compatibility condition $\w_{kj} J_i^k+\w_{ik} J_j^k=0$ we obtain the following form of the matrix
$$
J=\left[ \begin {array}{cccccc} {\it \psi_{11}}&{\it \psi_{12}}&{\it \psi_{13}}&{\it \psi_{14}}&{\it \psi_{15}}&{\it \psi_{16}}\\ \noalign{\medskip}{\it \psi_{21}}&{
\it \psi_{22}}&{\it \psi_{23}}&{\it \psi_{24}}&{\it \psi_{25}}&{\frac {{\it \psi_{15}}}{1-
\lambda}}\\ \noalign{\medskip}{\it \psi_{31}}&{\it \psi_{32}}&{\it \psi_{33}}&{\it \psi_{34}}&{\frac { \left( 1-\lambda \right) {\it \psi_{24}}}{\lambda}}&{\frac {{\it \psi_{14}}}{\lambda}}\\
\noalign{\medskip}{\it \psi_{41}}&{\it \psi_{42}}&{\it \psi_{43}}&-{\it \psi_{33}}&{\frac { \left( \lambda-1 \right) {\it \psi_{23}}}{\lambda}}&-{\frac {{\it \psi_{13}}}{\lambda}}\\
\noalign{\medskip}{\it \psi_{51}}&{\it \psi_{52}}&{\frac {\lambda{\it \psi_{42}}}{1-\lambda}}&{\frac {\lambda{\it \psi_{32}}}{\lambda-1}}&-{\it \psi_{22}}&{\frac {{\it \psi_{12}}}{\lambda-1}}\\
\noalign{\medskip}{\it \psi_{61}}& \left( 1-\lambda \right) {\it \psi_{51}}&\lambda{\it \psi_{41}}&-\lambda{\it \psi_{31}}& \left(\lambda-1 \right) {\it \psi_{21}}&-{\it \psi_{11}}
\end {array} \right].
$$

Now consider the system of integrability equations $N_{ij}^k=J_i^l J_j^m C_{lm}^k -J_i^l J_m^k C_{lj}^m-J_j^l J_m^k C_{il}^m+ C_{ij}^k=0$.
Let's calculate, for example, the element $N_{56}^1$:
$$
N_{56}^1=J_5^l J_6^m C_{lm}^1-J_5^l J_m^1 C_{l6}^m-J_6^l J_m^1 C_{5l}^m+C_{56}^1=-J_6^l J_m^1 C_{5l}^m=-J_6^1 J_6^1 C_{51}^6=(J_6^1 )^2=(\psi_{16} )^2.
$$
We obtain that $N_{56}^1=0$ when $\psi_{16}=0$. Under this condition, we obtain
$$
N_{45}^1=J_4^l J_5^m C_{lm}^1-J_4^l J_m^1 C_{l5}^m-J_5^l J_m^1 C_{4l}^m +C_{45}^1=-J_4^l J_m^1 C_{l5}^m-J_5^l J_m^1 C_{4l}^m=
$$
$$
=-J_4^1 J_6^1 C_{15}^6-J_5^1 J_5^1 C_{41}^5=-J_5^1 J_5^1 C_{41}^5=J_5^1 J_5^1=(\psi_{15})^2.
$$
We obtain that $N_{45}^1=0$ when $\psi_{15}=0$.
Under this condition, we obtain analogically
 $N_{15}^3=(\psi_{14})^2$ and $N_{35}^2=(\psi_{25})^2$.
If $\psi_{14}=0$ and $\psi_{25}=0$, then $N_{23}^1=(\psi_{13})^2$.
If $\psi_{13}=0$, then $N_{26}^6=(\psi_{12})^2/(\lambda-1)$, $N_{51}^2=((\psi_{24})^2 (1-\lambda))/\lambda$, $N_{35}^6=((\psi_{23})^2 (1-\lambda))/\lambda$.
If $\psi_{12}=0$, $\psi_{24}=0$, $\psi_{23}=0$, then $N_{41}^3=(\psi_{34})^2=0$ with $\psi_{34}=0$.
Under this conditions, we obtain the following three simple equations: $N_{13}^4 =2\psi_{11}\psi_{33} +(\psi_{33})^2+1=0$, $N_{23}^5=(\psi_{22})^2 +2\psi_{22}\psi_{33}+1=0$, $N_{15}^6=(\psi_{11})^2 -2\psi_{11}\psi_{22}+1=0$.
Their solutions are easy to find: $\psi_{11}=1, \psi_{22}=1, \psi_{33}=-1$, or $\psi_{11}=-1, \psi_{22}=-1, \psi_{33}=1$.
If they are taken into account, then some elements of $N_{ij}^k$ are non-zero constants.
For example, in the first case, when $\psi_{11}=1$, $\psi_{22}=1$, $\psi_{33}=-1$, we get $N_{24}^6=4$.
Indeed, in respect that $J_2^1=J_4^1=J_4^3=J_4^2=0$ and $J_2^2=\psi_{22}=1$, $J_4^4=-\psi_{33}=1$, $J_6^6=-\psi_{11}=-1$, we obtain:
$$
N_{24}^6=J_2^l J_4^m C_{lm}^6-J_2^l J_m^k C_{l4}^m-J_4^l J_m^k C_{2l}^m +C_{24}^6=
J_2^1 J_4^5 C_{15}^6+J_2^5 J_4^1 C_{51}^6+J_2^2 J_4^4 C_{24}^6+J_2^4 J_4^2 C_{42}^6 -
$$
$$
-J_2^1 J_5^6 C_{14}^5 -J_2^2 J_6^6 C_{24}^6
-J_4^1 J_3^6 C_{21}^3-J_4^3 J_5^6 C_{23}^5-J_4^4 J_6^6 C_{24}^6+1=
$$
$$
=J_2^2 J_4^4-J_2^2 J_6^6-J_4^4 J_6^6+1= 1+1+1+1=4\neq 0.
$$

\end{proof}


\begin{thebibliography}{999}

\bibitem{Aleks}
Alekseevsky D.V., Medori C., Tomassini A. Homogeneous para-K\"{a}hler Einstein manifolds. Russ. Math. Surv. 2009, Vol. 64, no. 1, P. 1–43.

\bibitem{BG}
Benson C. and Gordon C.S. K\"{a}hler and symplectic structures on nilmanifold. Topology, 1988, Vol. 27, no. 4, 513--518.


\bibitem{Banaru}
Banaru M. A note on parak\"{a}hlerian manifolds. Kyungpook Math. J. 2003. Vol. 43. No. 1. P. 49–61.

\bibitem{CFU}
Cordero L. A., Fern\'andez M. and Ugarte L. Pseudo-K\"{a}hler metrics on six dimensional nilpotent Lie algebras. J. of Geom. and Phys., 2004, Vol. 50, pp. 115--137.

\bibitem{Gray}
Gray A. Curvature identities for Hermitian and almost Hermitian manifolds. Tohoku Math. J. 1976. Vol. 28, No. 4. P. 601–612.

\bibitem{Goze-Khakim}
Goze M., Khakimdjanov Y., Medina A. Symplectic or contact structures on Lie groups. Differential Geom. Appl. 2004, Vol. 21, no. 1, 41--54.

\bibitem{KN}
Kobayashi S. and  Nomizu K. Foundations of Differential Geometry, Vol. 1 and 2. Interscience Publ. New York, London. 1963.

\bibitem{Liberman}
Libermann P. Sur les structures presque paracomplexes. C. R. Acad. Sci. Paris. 1952. Vol. 234. P. 2517–2519.

\bibitem{Rizza}
Rizza G. B. Varieta parak\"ahleriane. Ann. Mat. Pura Appl. 1974.Vol. 98. P. 47–61.

\bibitem{Schafer}
Sch\"{a}fer L. Canonical Ricci-flat nearly Parak\"ahlerian Manifolds. Ann. Global Anal. Geom. 2014. Vol. 45. No. 1. P. 11-24.

\bibitem{Smolen-11}
Smolentsev, N. K. Canonical pseudo-K\"ahler metrics on six-dimensional nilpotent Lie groups, Vestn. Kemerovsk. Gos. Univ., Ser. Mat., 2011, 3/1, No. 47, 155--168. (arXiv:1310.5395 [math.DG])

\end{thebibliography}
\end{document}